\documentclass[final,12pt,a4paper]{amsart}

 \title[A geometric model for the module category of a gentle algebra]{A geometric model for the module category of a gentle algebra}

\author{Karin Baur}
\email{baurk@uni-graz.at}

\author{Raquel Coelho Sim\~oes}
\email{rcoelhosimoes@campus.ul.pt}

\usepackage{amssymb}
\usepackage{latexsym}
\usepackage{amsmath}
\usepackage{mathrsfs}
\usepackage[all]{xy}
\usepackage{enumerate}	
\usepackage{enumitem}
\usepackage{paralist}
\usepackage{graphicx}

\usepackage{verbatim}

\usepackage{float}   
\usepackage{rotating} 

\usepackage{hyperref}
\newcommand{\harxiv}[1]{ \href{http://arxiv.org/abs/#1}{\texttt{arXiv:#1}}}
\newcommand{\hyref}[2]{ \hyperref[#2]{#1~\ref*{#2}} }

\usepackage[notref, notcite]{showkeys} 

\theoremstyle{plain}
\newtheorem{theorem}{Theorem}[section]
\newtheorem{thm}{Theorem}

\newtheorem{lemma}[theorem]{Lemma}
\newtheorem{corollary}[theorem]{Corollary}
\newtheorem{proposition}[theorem]{Proposition}

\theoremstyle{definition}
\newtheorem{remark}[theorem]{Remark}
\newtheorem{example}[theorem]{Example}
\newtheorem{definition}[theorem]{Definition} 

\DeclareMathAlphabet{\mathpzc}{OT1}{pzc}{m}{it}

\newcommand{\sA}{\mathsf{A}}

\newcommand{\sM}{\mathsf{M}}

\newcommand{\sP}{\mathsf{P}}

\newcommand{\sS}{\mathsf{S}}
\newcommand{\sT}{\mathsf{T}}

\newcommand{\mcC}{\mathcal{C}}

\newcommand{\mcP}{\mathcal{P}}

\newcommand{\mcT}{\mathcal{T}}

\newcommand{\Fac}{\mathsf{Fac}}
\newcommand{\Sub}{\mathsf{Sub}}

\renewcommand{\geq}{\geqslant}
\renewcommand{\leq}{\leqs}
\renewcommand{\phi}{\varphi}
\renewcommand{\epsilon}{\varepsilon}

\DeclareMathOperator{\Hom}{\mathsf{Hom}}

\renewcommand{\mod}[1]{\mathsf{mod}(#1)}

\newcommand{\kk}{{\mathbf{k}}}

\newcommand{\leqs}{\leqslant}

\newcommand{\arc}[1]{\mathtt{#1}}

\entrymodifiers={+!!<0pt,\fontdimen22\textfont2>}

\newcommand{\arup}{\ar@/^/[u]}
\newcommand{\ardn}{\ar@/^/[d]}

\newcommand{\coloneqq}{\mathrel{\mathop:}=}

\usepackage{tikz,tkz-euclide}
\usepackage{pgfplots}
        \usetikzlibrary{intersections}
	\usetikzlibrary{decorations.pathmorphing}
        \usetikzlibrary{topaths}
        \usetikzlibrary{automata,arrows}
        \usetikzlibrary{shapes.geometric} 
        \usetkzobj{all}
        \usetikzlibrary{backgrounds}

\begin{document}

\begin{abstract}
In this article, gentle algebras are realised as tiling algebras, which are associated to partial triangulations of unpunctured surfaces with marked points on the boundary. This notion of tiling algebras generalise the notion of Jacobian algebras of triangulations of surfaces and the notion of surface algebras. We use this description to give a geometric model of the module category of any gentle algebra. 
\end{abstract}



\maketitle

\addtocontents{toc}{\protect{\setcounter{tocdepth}{-1}}}  
\section*{Introduction} 
\addtocontents{toc}{\protect{\setcounter{tocdepth}{1}}}   

Gentle algebras are an important class of finite dimensional algebras of tame representation type. Gentle algebras are Gorenstein~\cite{GR}, closed under tilting and derived equivalence~\cite{S, SZ}, and they are ubiquitous, for instance they occur in contexts such as Fukaya categories~\cite{HKK}, dimer models~\cite{B}, enveloping algebras of Lie algebras~\cite{HK} and cluster theory. As such, there has been widespread interest in this class of algebras; see for example~\cite{BDMTY, SC, CPS, PPP} for recent developments in this area. 

In the context of cluster theory, gentle algebras arise as $m$-cluster-tilted algebras, $m$-Calabi-Yau tilted algebras (see e.g.~\cite{G-E}), and as Jacobian algebras associated to triangulations of unpunctured surfaces with marked points on the boundary~\cite{ABCP, LF}. 

The combinatorial geometry of the surfaces plays an important role in cluster-tilting theory, and in representation theory in general. It was used to give a geometric model of categories related to hereditary algebras of several types (see e.g.~\cite{BaM, BaT, CCS, L, Tor}), and of generalised cluster categories and module categories of Jacobian algebras~\cite{BZ}. The combinatorics of the surfaces is also useful to give a description of certain $m$-cluster-tilted algebras in terms of quivers with relations, and to study their derived equivalence classes  (see e.g.~\cite{Gu, Gu1, Mu}). On the other hand, geometric models associated to certain categories have also been used to tackle several representation-theoretic questions on those categories, such as the classification of torsion pairs (see e.g.~\cite{BBM, CSP, HJR}).

In this article, we realise gentle algebras as {\it tiling algebras}, by considering partial triangulations of the surfaces mentioned above. These algebras were already considered in~\cite{CSPs, GM} in the case where the surface is a disc, and they are a natural generalisation of Jacobian algebras and surface algebras~\cite{DRS}. See~\cite{Demonet} for other types of algebras related to partial triangulations. Viewing gentle algebras as tiling algebras allows us to construct a geometric model for their module categories. The main results of this article are stated as follows:

\begin{thm}[Proposition~\ref{prop:tilingendalgebra} and Theorem~\ref{thm:tilinggentle}]\label{thmA}
Let $A$ be a finite dimensional algebra. Then the following are equivalent. 
\begin{compactenum}
\item $A$ is a gentle algebra. 
\item $A$ is a tiling algebra.
\item $A$ is the endomorphism algebra of a partial cluster-tilting object of a generalised cluster category. 
\end{compactenum}
\end{thm}

\begin{thm}[Theorem~\ref{thm:permissiblearcsstrings}, Proposition~\ref{prop:bandsclosedcurves} and Theorem~\ref{thm:irreduciblepivots}]\label{thmB:model}
Let $A_{\sP}$ be the tiling algebra associated to a partial triangulation $\sP$ of an unpunctured surface $\sS$ with marked points $M$ in its boundary. There are bijections between:
\begin{compactenum}
\item the equivalence classes of non-trivial permissible arcs in $(\sS,M)$ and non-zero strings of $A_{\sP}$;
\item permissible closed curves and powers of band modules; 
\item certain moves between permissible arcs and irreducible morphisms of string modules.
\end{compactenum}
\end{thm}
In the cases when the partial triangulation is either a triangulation or a cut of a triangulation, our model in Theorem~\ref{thmB:model} recovers the one given for surface algebras~\cite{DRS} (see also~\cite{ABCP,BZ}). We note that recently, a geometric description for the bounded derived category of a gentle algebra (up to shift) was introduced in~\cite{OPS} (see also~\cite{LP}, for the case of homologically smooth graded gentle algebras). The special case of discrete derived categories was previously studied in~\cite{Broomhead}. 

We also give a simple description, in terms of the partial triangulations of surfaces, of all the representation-finite gentle algebras. Theorem~\ref{thmB:model} (1) and (3) gives a complete geometric description of this class of gentle algebras, and examples include gentle algebras with no double arrows nor loops, and whose cycles are oriented and with radical square zero (these arise from discs) and derived-discrete algebras (these arise from annuli). 

This paper is organised as follows. In Section \ref{sec:gentlebackground}, we give the necessary background on gentle algebras. In Section \ref{sec:gentleandtilingalgebras} we introduce the notion of tiling algebra and give the proof of Theorem~\ref{thmA}. Section \ref{sec:geometricmodel} is devoted to the geometric model for the module category of a tiling algebra, and it includes a geometric description of morphisms between indecomposable modules.  

\section{Gentle algebras: basic notions and properties}\label{sec:gentlebackground}

In this section, we will recall the definition of gentle algebras and basic properties of their module categories. The main reference is~\cite{Butler-Ringel}. 

Let $\kk$ denote an algebraically closed field, $Q$ a quiver with set of vertices $Q_0$ and set of arrows $Q_1$, $I$ an admissible ideal of $\kk Q$, and $A$ the bound quiver algebra $\kk Q/I$. Given $a \in Q_1$, we denote its source by $s(a)$ and its target by $t(a)$. We will read paths in $Q$ from left to right, and $\mod A$ will denote the category of right $A$-modules. 

\begin{definition}\cite{AS}
A finite dimensional algebra $A$ is {\it gentle} if it admits a presentation $A = \kk Q/I$ satisfying the following conditions: 
\begin{compactenum}[(G1)]
\item Each vertex of $Q$ is the source of at most two arrows and the target of at most two arrows.
\item For each arrow $\alpha$ in $Q$, there is at most one arrow $\beta$ in $Q$ such that $\alpha \beta \not\in I$, and there is at most one arrow $\gamma$ such that $\gamma \alpha \not\in I$. 
\item For each arrow $\alpha$ in $Q$, there is at most one arrow $\delta$ in $Q$ such that $\alpha \delta \in I$, and there is at most one arrow $\mu$ such that $\mu \alpha \in I$. 
\item $I$ is generated by paths of length 2.
\end{compactenum}
\end{definition}

Throughout this section, $A = \kk Q/ I$ denotes a gentle algebra. To each arrow $a \in Q_1$,  we associate a formal inverse $a^{-1}$, with $s(a^{-1}) = t(a)$ and $t(a^{-1}) = s(a)$. The set of formal inverses of arrows in $Q_1$ is denoted by $Q^{-1}$. We can extend this to any path $p = a_1 \ldots a_n$ in $Q$, by setting 
$p^{-1} = a_n^{-1} \ldots a_1^{-1}$. 

A {\it string} $w$ is a reduced walk in the quiver which avoids relations, i.e. $w = a_1 \ldots a_n$, with $a_i \in Q \cup Q^{-1}$, for which there are no subwalks of the form $aa^{-1}$ or $a^{-1} a$, for some $a \in Q_1$, nor subwalks $ab$ such that $ab \in I$ or $(ab)^{-1} \in I$. The {\it length} of a string $w=a_1\ldots a_n$ is $n$. For each vertex $v \in Q_0$, there are two {\it trivial strings} $1_v^+$ and $1_v^-$. The vertex $v$ is both the source and target of each corresponding trivial string, and the formal inverse acts by swapping the sign, i.e. $(1_v^{\pm})^{-1} = 1_v^{\mp}$. For technical reasons, we also consider the empty string, also called the {\it zero string}. We will denote it by $w=0$.

A string $w= a_1 \ldots a_n$ is {\it direct} ({\it inverse}, resp.) if $a_i \in Q_1$ ($a_i \in Q_1^{-1}$, resp.), for all $i= 1, \ldots n$. 

A string is {\it cyclic} if its source and target coincide. A {\it band} is a cyclic string $b$ for which each power $b^n$ is a string, but $b$ itself is not a proper power of any string.

Each string $w$ in $A$ defines a {\it string module} $M(w) \in \mod A$. The underlying vector space of $M(w)$ is obtained by replacing each vertex of $w$ by a copy of the field $\kk$, and the action of an arrow $a$ on $M(w)$ is the identity morphism if $a$ is an arrow of $w$, and is zero otherwise. If $w =0$, then $M(w)=0$. Note that $M(w) \simeq M(v)$ if and only if $v = w$ or $v = w^{-1}$. 

Each band $b = a_1 \cdots a_n$ (up to cyclic permutation and inversion) defines a family of band modules $M(b, n, \varphi)$, where $n \in \mathbb{N}$ and $\varphi$ is an indecomposable automorphism of $k^n$. Here, each vertex of $b$ is replaced by a copy of the vector space $k^n$, and the action of an arrow $a$ on $M(b, n, \varphi)$ is induced by identity morphisms if $a = a_i$, for $i = 1, \ldots, n-1$ and by $\varphi$ if $a = a_n$. All string and band modules are indecomposable, and every indecomposable $A$-module is either a string module of a band module (cf.~\cite{Butler-Ringel}). 

Given two strings $v= a_1 \cdots a_n, w= b_1 \cdots b_m$ of length at least 1, the composition $vw$ is defined if $vw = a_1 \cdots a_n b_1 \cdots b_m$ is a string. 

In order to define composition of strings with trivial strings, and avoid ambiguity in the description of irreducible morphisms between string modules, we need to consider two sign functions $\sigma, \epsilon\colon Q_1 \rightarrow \{-1,1\}$ satisfying the following conditions:
\begin{compactitem}
\item If $b_1 \neq b_2 \in Q_1$ are such that $s(b_1) = s(b_2)$, then $\sigma (b_1) = - \sigma (b_2)$. 
\item If $a_1 \neq a_2 \in Q_1$ are such that $t(a_1) = t(a_2)$, then $\epsilon (a_1) = - \epsilon (a_2)$.
\item If $ab$ is not in a relation, then $\sigma (b) = - \epsilon (a)$. 
\end{compactitem}
We refer the reader to~\cite[p. 158]{Butler-Ringel}, for an algorithm for choosing these functions. We can extend the maps $\sigma, \epsilon$ to all strings as follows. Given an arrow $b \in Q_1$, we have $\sigma (b^{-1}) = \epsilon (b)$, and $\epsilon (b^{-1}) = \sigma (b)$. If $w = a_1 \cdots a_n$ is a string of length $n \geq 1$, $\sigma (w) = \sigma (a_1)$ and $\epsilon (w) = \epsilon (a_n)$. Finally, $\sigma (1_v^\pm) = \mp 1$ and $\epsilon (1_v^\pm) = \pm 1$. 

If $v$ is a string and $w = 1_x^\pm$, then $vw$ is defined if $t(v) = x$ and $\epsilon (v) = \pm 1$. Analogously, if $v = 1_x^\pm$ and $w$ is a string, then $vw$ is defined if $s(w) = x$ and $\sigma (w) = \mp 1$. In a nutshell, composition of strings $vw$ of a gentle algebra is defined if $\sigma(w) = - \epsilon (v)$. 

Now we recall the description of irreducible morphisms between string $A$-modules, given in~\cite{Butler-Ringel}. Let $w$ be a string and $M(w)$ be the corresponding string module. The irreducible morphisms starting at $M(w)$ can be described in terms of modifying the string $w$ on the left or on the right. 

{\it Case 1:} There is an arrow $a \in Q_1$ such that $a w$ is a string. Then let  $v$ be the maximal direct string starting at $s(a)$, and define $w_\ell$ to be the string $v^{-1} a w$.

{\it Case 2:} There is no arrow $a \in Q_1$ such that $a w$ is a string. Then, either $w$ is a direct string, or we can write $w = v a^{-1} w'$, where $a \in Q_1$, $w'$ is a string, and 
$v$ is a maximal direct substring starting at $s(w)$. In the former case, we define $w_\ell$ to be the zero string, and in the latter case $w_\ell \coloneqq w'$. 

In the literature, $w_\ell$ is said to be obtained from $w$ by either {\it adding a hook on $s(w)$} (in Case 1) or {\it removing a cohook from $s(w)$} (in Case 2). 

The definition of $w_r$ is dual. One can check that $(w_\ell)_r = (w_r)_\ell$, and we shall denote this string by $w_{r,\ell}$. Figures~\ref{fig:hooks1} and \ref{fig:hooks2} illustrates these concepts. 

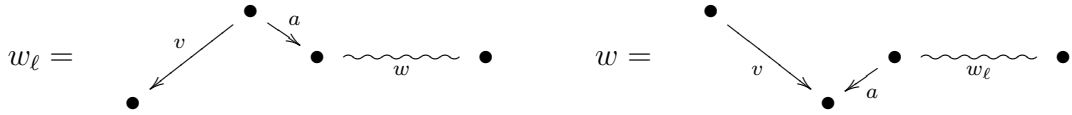
\begin{figure}
\xymatrix@C-1pc@R=4pt{
 &&& \bullet\ar[lldd]_{v}\ar[rd]^{a} && && 
    &&&  \bullet\ar[rrdd]_{v}  \\
w_{\ell}=&&  && \bullet\ar@{~}[rrr]_{w} &&& \bullet 
     &&w=&&&&\bullet\ar[ld]^a\ar@{~}[rrr]_{w_{\ell}} &&&\bullet  \\
 & \bullet  &&&&&& 
   & &&   & &\bullet  && 
}
\caption{Left: adding a hook on $s(w)$. Right: removing a cohook from $s(w)$.}
\label{fig:hooks1}
\end{figure} 

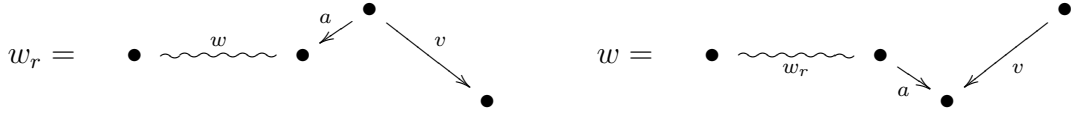
\begin{figure}
\xymatrix@C-1pc@R=4pt{ 
 &&&&& \bullet\ar[ld]_{a}\ar[rrdd]^{v}  && 
 && &
   &&&&&& \bullet\ar[lldd]^{v}   \\
 w_r=&\bullet &&& \bullet\ar@{~}[lll]_{w} &&   &
 &&w=& 
   \bullet\ar@{~}[rrr]_{w_r} &&& \bullet\ar[rd]_a &&&  \\ 
 &&&&&&& \bullet  
 & && 
   &&& &  \bullet  && 
}
\caption{Left: adding a hook on $t(w)$. Right: removing a cohook from $t(w)$.}
\label{fig:hooks2}
\end{figure} 
 
\begin{proposition}\label{prop:BRmaps}\cite{Butler-Ringel}
Let $w$ be a string of a gentle algebra $A$ and $M(w)$ be the corresponding string module. 
\begin{compactenum}
\item There are at most two irreducible morphisms starting at $M(w)$, namely $M(w) \rightarrow M(w_\ell)$ and $M(w) \rightarrow M(w_r)$. These maps are either the natural inclusion (in case 1), or the natural projection (in case 2). 
\item If $M(w)$ is not injective, then the Auslander-Reiten sequence starting at $M(w)$ is given by $0 \rightarrow M(w) \rightarrow M(w_\ell) \oplus M(w_r) \rightarrow M(w_{r, \ell}) \rightarrow 0$. In particular, $\tau^{-1} (M(w)) = M(w_{r,\ell})$.
\end{compactenum}
\end{proposition}

This proposition actually holds more generally for string algebras. However, we only state it for gentle algebras, since this is the case we are interested in. 
The remaining irreducible morphisms are between band modules of the form $M(b,n,\varphi)$ and $M(b,n+1,\varphi)$. Band modules lie in homogeneous tubes, and so the Auslander-Reiten translate acts on them as the identity morphism.

\section{Gentle algebras as tiling algebras}\label{sec:gentleandtilingalgebras}

Throughout this paper, $\sS$ denotes an unpunctured, connected oriented Riemann surface with non-empty boundary and with a finite set $M$ of marked points on the boundary. We assume that, if $\sS$ is a disc, then $M$ has at least four marked points. The pair $(\sS,M)$ is simply called a marked surface. We allow {\it unmarked boundary components of $\sS$}, i.e. boundary components with no marked points in $M$.

An {\it arc} in $(\sS, M)$ is a curve $\gamma$ in $\sS$ satisfying the following properties:
\begin{compactitem}
\item The endpoints of $\gamma$ are in $M$.
\item $\gamma$ intersects the boundary of $\sS$ only in its endpoints. 
\item $\gamma$ does not cut out a monogon or a digon. 
\end{compactitem}

We always consider arcs up to homotopy relative to their endpoints. We call a collection of arcs $\sP$ that do not intersect themselves or each other in the interior of $\sS$ a {\it partial triangulation}. A triple $(\sS, M, \sP)$ satisfying the conditions above is a {\it tiling}.

Before giving the definition of tiling algebra associated to $(\sS, M, \sP)$, we need to introduce some terminology and notation. We follow the nice exposition given in~\cite{DRS}. 

An arc $\arc{v}$ in $(\sS, M,\sP)$ has two {\it end segments}, defined by trisecting the arc and deleting the middle piece. We denote these two end segments by $\overline{\arc{v}}$ and $\overline{\overline{\arc{v}}}$. 

Given a marked point $p$ in $M$, let $m', m''$ be two points in the same boundary component of $p$ such that $m', m'' \not\in M$ and $p$ is the only marked point lying in the boundary segment $\delta$ between $m'$ and $m''$. Draw a curve $c$ homotopic to $\delta$ but lying in the interior of $\sS$ except for its endpoints $m'$ and $m''$. The {\it complete fan at $p$} is the sequence $\arc{v}_1,  \ldots, \arc{v}_k$ of arcs in $\sP$ which $c$ crosses in the clockwise order. Any subsequence $\arc{v}_{i}, \arc{v}_{i+1}, \ldots, \arc{v}_j$ is called a {\it fan at $p$}. The arc $\arc{v}_i$ ($\arc{v}_j$, resp.) is said to be the {\it start} ({\it end}, resp.) of this fan. If the marked point $p$ is not the endpoint of any arc in $\sP$, then the (complete) fan at $p$ is called an {\it empty fan}. 

Every arc $\arc{v}$ in $\sP$ lies in either one non-empty complete fan, in case $\arc{v}$ is a loop, i.e. both endpoints coincide, or two distinct non-empty complete fans. 

\begin{definition}
The {\it tiling algebra} $A_{\sP}$ associated to the partial triangulation $\sP$ of $(\sS,M)$ is the bound quiver algebra $A_\sP = \kk Q_\sP / I_\sP$, where $(Q_\sP, I_\sP)$ are described as follows: 
\begin{compactenum}
\item The vertices in $(Q_\sP)_0$ are in one-to-one correspondence with the arcs in $\sP$. 
\item There is an arrow $a\colon a_s \rightarrow a_t$ in $(Q_\sP)_1$ if the arcs $\arc{a}_s$ and $\arc{a}_t$ share an endpoint $p_a \in M$ and $\arc{a}_t$ is the immediate successor of $\arc{a}_s$ in the complete fan at $p_a$.  
\item $I_\sP$ is generated by paths $ab$ of length 2 which satisfy one of the following conditions: 
\begin{compactitem}
\item $b = a$, i.e. the arcs corresponding to the sources and targets of $a$ and $b$ coincide and it is a loop;
\item $b \neq a$ and either $p_a \neq p_b$ or we are in the situation presented in Figure \ref{fig:relations}, 
with $\arc{a}$ a loop. 
\end{compactitem}
\end{compactenum}
\end{definition}

\begin{figure}[H]
\includegraphics[width=10cm]{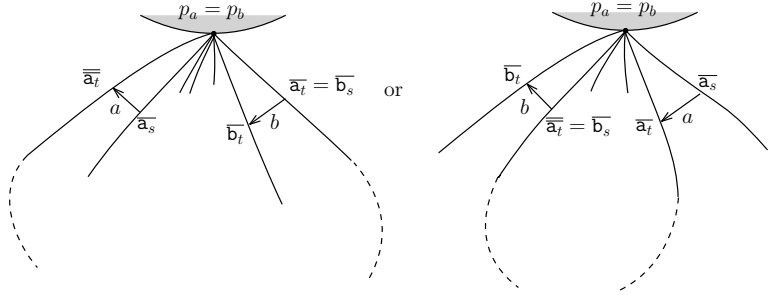}
\caption{Case $ab=0$ when $p_a=p_b$.}
\label{fig:relations}
\end{figure}

\begin{remark}
If $a, b \in (Q_\sP)_1$ with $t(b) = s(a)$ are such that $ab \in I_\sP$, then the arcs corresponding to $s(a), t(b) = s(a)$ and $t(b)$ bound the same region in $(\sS, M, \sP)$. However, the converse is not true in general. 
\end{remark}

\begin{lemma}
Every oriented cycle in $A_\sP$ of length $\geq 2$ has a zero relation. 
\end{lemma}
\begin{proof}
Let $c = \xymatrix{v_1 \ar[r]^{a_1} & v_2 \ar[r]^{a_2} & \cdots \ar[r]^{a_{k-1}} &  v_k \ar[r]^{a_k} &  v_1}$ be an oriented cycle of length $k \geq 2$. Let $p_{i}$ be the marked point associated to arrow $a_i$, for $i = 1, \ldots, k$. If $a_i = a_{i+1}$ or $p_{i} \neq p_{i+1}$ for some $i =1, \ldots k+1$ (where $k+1 =1$), then $a_i a_{i+1} \in I_{\sP}$, by definition. So assume $a_i \neq a_{i+1}$ and $p \coloneqq p_i = p_{i+1}$, for all $i = 1, \ldots k+1$. Then the cycle $c$ corresponds to a fan at $p$ with $\arc{v}_1$ as both the start and end. In other words, $\arc{v}_1$ is a loop, and by definition of $I_\sP$, it follows that $a_k a_1 \in I_\sP$, which finishes the proof.  
\end{proof}

\begin{corollary}
Every tiling algebra is finite-dimensional. 
\end{corollary}

A {\it triangulation of $(\sS,M)$} is a partial triangulation cutting the surface into triangles. Note that not every surface admits a triangulation. 

Let $(\sS, M')$ be a new marked surface obtained from $(\sS,M, \sP)$ by adding a marked point in each unmarked boundary component. Note that these new marked points do not play a role in the definition of a tiling algebra, i.e. the tiling algebras associated to $(\sS, M, \sP)$ and $(\sS, M', \sP)$ are isomorphic. On the other hand, the partial triangulation $\sP$ in $(\sS, M')$ can be completed to a triangulation $\sT$ (cf.~\cite[Lemma 2.13]{FST}). 

Let $A_\sT = \kk Q_\sT / I_\sT$ be the Jacobian algebra associated to $\sT$ (see~\cite{ABCP} for more details). In particular, each vertex $v$ of $Q_\sT$ corresponds to an arc $\arc{v}$ of $\sT$, and the relations in $A_\sT$ are precisely the compositions of any two arrows in an oriented 3-cycle in $Q_\sT$, which correspond to internal triangles in $\sT$, i.e. triangles whose three sides are arcs in $\sT$.  

\begin{definition}
Let $A_{\sP, \sT} = \kk Q_{\sP,\sT} / I_{\sP, \sT}$ be the algebra obtained from $(Q_\sT, I_\sT)$ as follows:

Vertices: $(Q_{\sP,\sT})_0 = (Q_\sT)_0 \setminus \{u \mid \arc{u} \in \sT \setminus \sP\}$. In other words, the vertices are in one-to-one correspondence with the arcs in $\sP$. 

Arrows: For each direct string $v \rightarrow u_1 \rightarrow \cdots \rightarrow u_k \rightarrow v'$ in $Q_\sT$, where $\arc{v}, \arc{v}' \in \sP$ and $\arc{u}_1, \ldots, \arc{u}_k \in \sT \setminus \sP$, we have an arrow $v \rightarrow v'$ in $Q_{\sP,\sT}$. All the arrows in $Q_{\sP,\sT}$ are obtained in this way. 

Relations: Let $\xymatrix{v \ar[r]^a & v' \ar[r]^b & v''}$ be a path of length 2 in $Q_{\sP,\sT}$. This means we have two direct strings
$\xymatrix@C-1pc{v \ar[r] & u_1 \ar[r] & \cdots \ar[r] & u_k \ar[r]^\alpha & v'}$  and 
$\xymatrix@C-1pc{v' \ar[r]^\beta & u'_1 \ar[r] & \cdots \ar[r] &  u'_\ell \ar[r] & v''}$ in $Q_\sT$, 
with $\arc{u}_i, \arc{u'}_j \in \sT \setminus \sP$ for all $i$, $j$, 
and $\arc{v}, \arc{v}', \arc{v}'' \in \sP$. If $\alpha \beta \in I_\sT$, then  $ab \in I_{\sP,\sT}$.
\end{definition}

\begin{lemma}\label{lem:tilingalgebraviatriangulations}
The tiling algebra $A_\sP$ is isomorphic to $A_{\sP, \sT}$, for any triangulation $\sT$ completing $\sP$ in $(\sS, M')$. 
\end{lemma}
\begin{proof}
It is clear that the set of vertices $(Q_\sP)_0$ coincides with $(Q_{\sP,\sT})_0$, which corresponds to the arcs in $\sP$. 

An arrow $a\colon v \rightarrow v'$ belongs to $(Q_{\sP,\sT})_1$ if and only if there is a direct string $v \rightarrow u_1 \rightarrow \cdots \rightarrow u_r \rightarrow v'$ in $Q_\sT$, where each $u_i$ corresponds to an arc in $\sT \setminus \sP$. This is equivalent to saying that the sequence $\arc{v}, \arc{u}_1, \ldots, \arc{u}_r, \arc{v'}$ is a fan at a marked point $p$ of $(\sS, M', T)$, which in turn is equivalent to saying that the sequence $v, v'$ is a fan at the marked point $p$ of $(\sS,M,\sP)$, i.e. that $a\colon v \rightarrow v' \in (Q_P)_1$. Hence $Q_\sP = Q_{\sP,\sT}$. 

Finally, we have $ab \in I_{\sP,\sT}$ if and only if the following is a subquiver of $Q_\sT$: 
\[
\xymatrix{
s(a) \ar[r] & u_{1} \ar@{~>}[rr]  & & u_{r} \ar[r]^{\alpha} & t(a) \ar[r]^{\beta} 
 & u'_{1} \ar@/^1pc/[ll]^{\gamma} \ar@{~>}[rr] & & u'_{s} \ar[r] & t(b)
 },
\]
where the oriented 3-cycle has radical square zero.  We can deduce that this is equivalent to saying that $ab \in I_\sP$, by noting that $t(\alpha) = t(a) = s(b) = s(\beta)$, $p_a = p_\alpha$ and $p_b = p_\beta$, and by having in mind the possible internal triangles in $\sT$ that give rise to the oriented 3-cycle in the subquiver above.
\end{proof}

\begin{remark}
If $\sP$ is a triangulation, then our notion of tiling algebra recovers the concept of Jacobian algebra. Furthermore, our notion generalises the notion of surface algebra, which is defined for cuts of triangulations (cf.~\cite{DRS}).
\end{remark}

We will now state two neat consequences of Lemma~\ref{lem:tilingalgebraviatriangulations}. Assume $(\sS,M)$ is such that each boundary component of $\sS$ contains a marked point in $M$. Then one can associate a category, called generalised cluster category $\mcC_{\sS,M}$, to $(\sS, M)$. Moreover, every indecomposable object in $\mcC_{\sS,M}$ can be described in terms of curves in the surface. In particular, there is a bijection between (partial) triangulations $\sP$ of $(\sS,M)$ and (partial) cluster-tilting objects $\mcP$ in $\mcC_{\sS,M}$. Given a triangulation $\sT$ of $(\sS,M)$, the endomorphism algebra $B_\mcT$ of the corresponding cluster-tilting object $\mcT$ is known to be isomorphic to the Jacobian algebra $A_\sT$. More details can be found in~\cite{Amiot, ABCP, BZ, LF}. 

\begin{proposition}\label{prop:tilingendalgebra}
Let $A$ be a finite dimensional algebra. The following conditions are equivalent:
\begin{compactenum}
\item $A$ is a tiling algebra.
\item $A$ is isomorphic to the endomorphism algebra of a partial cluster-tilting object in a generalised cluster category. 
\end{compactenum}
\end{proposition}
\begin{proof}
By Lemma~\ref{lem:tilingalgebraviatriangulations}, the algebra $A$ is a tiling algebra if and only if $A = A_{\sP,\sT}$, for some partial triangulation $\sP$ of a marked surface $(\sS,M)$ which admits a triangulation, and some triangulation $\sT$ completing $\sP$. By above, partial triangulations $\sP$ of marked surfaces $(\sS, M)$ admitting triangulations, are in bijection with partial cluster-tilting objects $\mcP$ in the generalised cluster category $(\sS,M)$. Denote the corresponding endomorphism algebra by $B_\mcP$. It is then enough to show that $A_{\sP,\sT} \simeq B_\mcP$. 

Since the Jacobian algebra $A_\sT$ is isomorphic to the endomorphism algebra $B_\mcT$ of the cluster-tilting object $\mcT$, we know that the vertices of the quiver $A_\sT$ are in one-to-one correspondence with the indecomposable summands of $\mcT$, and each non-zero path in the quiver of $A_\sT$ correspond to a non-zero morphism in $\mcC_{\sS,M}$ between the indecomposable summands associated to the endpoints of the path. We use the same notation for vertices in $A_\sT$ and indecomposable objects in $\mcC_{\sS,M}$, and for arrows in $A_\sT$ and corresponding morphisms in $\mcC_{\sS,M}$. 

Given two vertices $v$ and $v'$ of $A_\sT$ corresponding to two indecomposable summands of $\mcP$, each arrow $v \rightarrow v'$ in $B_\mcP$ corresponds a non-zero morphism from $v$ to $v'$ which does not factor through any other summand of $\mcP$. By above, this corresponds to a direct string $v \rightarrow u_1 \rightarrow \cdots \rightarrow u_r \rightarrow v'$ in $Q_\sT$, where each $\arc{u}_i$ lies in $\sT \setminus \sP$. Thus, the set of arrows of $(B_\mcP)$ coincides with $(Q_{\sP,\sT})_1$. 

Similary, a path $v_1 \rightarrow v_2 \rightarrow \cdots \rightarrow v_k$ in $B_\mcP$ corresponds to a zero morphism in $\mcC_{\sS,M}$ if and only if there are direct strings $v_i \rightarrow u_{i1} \rightarrow u_{i2} \rightarrow \cdots \rightarrow u_{ii_r} \rightarrow v_{i+1}$, for each $i = 1, \ldots, k-1$, such that each $\arc{u}_{ij} \in \sT \setminus \sP$, and the composition of all these strings yields a relation. Given the notion of the Jacobian algebra $A_\sT$, this is equivalent to saying that there is a relation in $I_\sT$ of the form $u_{ii_r} \rightarrow v_{i+1} \rightarrow u_{i+1,1}$. Thus, the ideal of $B_\mcP$ coincides with that of $A_{\sP,\sT}$, which finishes the proof. 
\end{proof}

Proposition~\ref{prop:tilingendalgebra} generalises~\cite[Theorem 3.7]{DRS}. The following proposition is a generalisation of~\cite[Lemma 2.5]{ABCP} and~\cite[Theorem 3.8]{DRS}. 

\begin{proposition}
Every tiling algebra $A_\sP$ is gentle. 
\end{proposition}
\begin{proof}
Let $\sP$ be a partial triangulation of $(\sS, M')$, $\sT$ be a triangulation completing $\sP$, and $A_\sT$ be the corresponding Jacobian algebra, which is known to be gentle~\cite[Lemma 2.5]{ABCP}.  

By Lemma~\ref{lem:tilingalgebraviatriangulations}, $A_\sP \simeq A_{\sP, \sT}$. Let $v$ be a vertex in $A_{\sP,\sT}$, and suppose, for a contradiction, that there are three distinct arrows starting at $v$. This means there are three distinct direct strings in $A_\sT$ starting at $v$. These direct strings are subpaths of three distinct non-trivial maximal direct strings in $A_\sT$ which contain the vertex $v$. However, since $A_\sT$ is gentle, every vertex lies in at most two non-trivial maximal direct strings, so we reached a contradiction. Similarly, we can prove that every vertex in $A_{\sP,\sT}$ is the target of at most two arrows. 

Consider the following subquiver of $Q_{\sP,\sT}$:
\[
\xymatrix@R-2pc{
                &   &  v_3\\
v_1 \ar[r]^\alpha & v_2 \ar[ur]^\beta \ar[dr]_\gamma \\
                &   & v_4.} 
\]
Then, we have the following subquiver of $Q_\sT$:
\[
\xymatrix@R-2pc{
           &              &            &          &            & v_3\\
           &              &            &          & u' \ar@{~>}[ur] \\
v_1 \ar[r] & \cdots \ar[r] & u \ar[r]^a & v_2 \ar[ur]^{b} \ar[dr]_c  \\
           &               &           &          & u'' \ar@{~>}[dr]  \\
           &               &           &          &            & v_4}.       
\]
Since $A_\sT$ is gentle, either $ab = 0$ or $ac = 0$, which implies that either $\alpha \beta = 0$ or $\alpha \gamma = 0$. A similar argument proves the dual. Finally, (G4) holds by definition.
\end{proof}

Conversely, we can associate to each gentle algebra $A$ a certain graph called {\it marked Brauer graph} $\Gamma_A$. We refer the reader to~\cite[Definition 1.8]{OPS} for details (see also~\cite{Schroll15}). This graph can be embedded in a surface $S_A$ with boundary (cf.~\cite{Labourie, OPS}), in such a way that the vertices of $\Gamma_A$ correspond to marked points $M_A$ in the boundary, and the edges are curves in the surface which do not intersect themselves or each other, except at the endpoints. The authors in~\cite{OPS} define a {\it lamination} of this embedding, which they use to recover the original algebra $A$. Let $M'$ be the set of marked points obtained from $M_A$ by adding a marked point in each boundary segment homotopic to an edge of the embedding of $\Gamma_A$ in $S_A$. Thus, $\Gamma_A$ can be realised as a partial triangulation of $(S_A, M')$ and one can check that the tiling algebra associated to $(S_A, M', \Gamma_A)$ is isomorphic to the algebra associated to the lamination mentioned above, and hence it is isomorphic to the gentle algebra we started with. This concludes the proof of the following theorem.

\begin{theorem}\label{thm:tilinggentle}
An algebra is gentle if and only if it is a tiling algebra. 
\end{theorem} 

%
\section{The geometric model of the module category of a tiling algebra}\label{sec:geometricmodel}
%

In order to characterise the module category of a gentle algebra geometrically, we will consider tilings $(\sS,M,\sP)$ where the partial triangulation $\sP$ divides $\sS$ into a collection of regions (also called {\it tiles}) of the following list: 

\begin{compactitem}
\item $m$-gons, with $m \geq 3$ and whose edges are arcs of $\sP$ and possibly boundary segments, and whose interior contains no unmarked boundary component of $\sS$;
\item one-gons, i.e. loops, with exactly one unmarked boundary component in their interior - we call these {\it tiles of type I};
\item two-gons with exactly one unmarked boundary component in their interior - we call these {\it tiles of type II} (see Figure \ref{fig:regionsIandII}).
\end{compactitem} 

\begin{figure}[H]
\includegraphics[height=3cm]{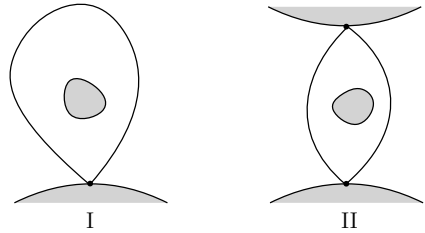}
\caption{Tiles of type I and II.}
\label{fig:regionsIandII}
\end{figure}

The fact that every gentle algebra $A$ can be realised as a tiling algebra associated to a tiling $(\sS, M, \sP)$ satisfying this setup, follows from~\cite[Proposition 1.12]{OPS}, which gives a description of the regions of the surface $(S_A, M_A, \Gamma_A)$. The tiling $(\sS, M, \sP)$ is obtained from $(S_A, M_A, \Gamma_A)$ (recall this notation from end of Section~\ref{sec:gentleandtilingalgebras}) by adding a marked point in each boundary segment homotopic to a curve of $\Gamma_A$, and by removing every unmarked boundary component lying in the interior of an $m$-gon, for $m \geq 3$. It is clear from the definition, that these changes do not affect the notion of tiling algebra.

From now on, $A_\sP$ denotes the tiling algebra associated to a tiling $(\sS, M, \sP)$ satisfying the conditions described above. 

\subsection{Permissible arcs and closed curves} 

Not every curve in the surface, with endpoints in $M$ will correspond to an indecomposable $A_\sP$-module. Moreover, non-homotopic curves can correspond to the same indecomposable module. Hence, we need the notion of permissible curves and of equivalence between these curves. 

\begin{definition}
\begin{compactenum}
\item A curve $\gamma$ in $(\sS,M, \sP)$ is said to {\it consecutively cross} $\arc{a}_1, \arc{a}_2 \in \sP$ if $\gamma$ crosses $\arc{a}_1$ and $\arc{a}_2$ in the points $p_1$ and $p_2$, and the segment of $\gamma$ between the points $p_1$ and $p_2$ does not cross any other arc in $\sP$. 
\item Let $B$ be an unmarked boundary component of $\sS$, $\gamma\colon [0,1] \rightarrow \sS$ be an arc in $(\sS,M)$, and write $\gamma = \gamma_1 \gamma' \gamma_2$, where $\gamma_1$ ($\gamma_2$, resp.) is the segment between $\gamma (0)$ ($\gamma (1)$, resp.) and the first (last, resp.) crossing with $\sP$. The {\it winding number of $\gamma_i$ around $B$}, where $i = 1, 2$, is the minimum number of times $\beta$ travels around $B$ in either direction, with $\beta$ lying in the homotopy class of $\gamma_i$.
\item An arc $\gamma$ in $(\sS,M, \sP)$ is called {\it permissible} if it satisfies the following two conditions: 
\begin{compactenum}
\item The winding number of $\gamma_i$ around an unmarked boundary component is either zero or one, for $i = 1, 2$.
\item If $\gamma$ consecutively crosses two (possibly not distinct) arcs $x$ and $y$ of $\sP$, then $x$ and $y$ have a common endpoint $p \in M$, and locally we have a triangle, as shown in Figure~\ref{fig:permissible}. 
\end{compactenum} 
\item A {\it permissible closed curve} is a closed curve $\gamma$ which satisfies condition (3)(b). 
\end{compactenum}
\end{definition}

\begin{figure}[H]
\includegraphics[height=2.5cm]{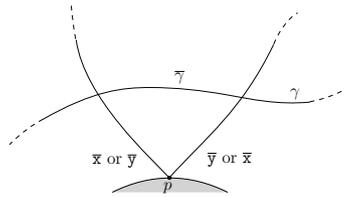}
\caption{$\gamma$ is a permissible arc.}\label{fig:permissible}
\end{figure}

Note that boundary segments and arcs in $\sP$ are considered to be permissible arcs. 

\begin{remark}
It follows from the definition that, to each consecutive crossing of a permissible arc with arcs in $\sP$ is associated an arrow of the tiling algebra. 
\end{remark}

\begin{example}
Figures \ref{fig:exnotpermissible1}, \ref{fig:exnotpermissible2} and \ref{fig:expermissible} show some examples of permissible and non-permissible arcs.
\begin{figure}[H]
\includegraphics[height=3cm]{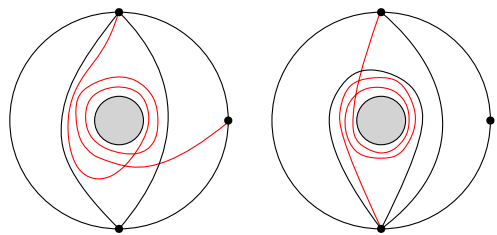}
\caption{Examples of arcs not satisfying condition (3)(a).} 
\label{fig:exnotpermissible1}
\end{figure}
\begin{figure}[H]
\includegraphics[height=3cm]{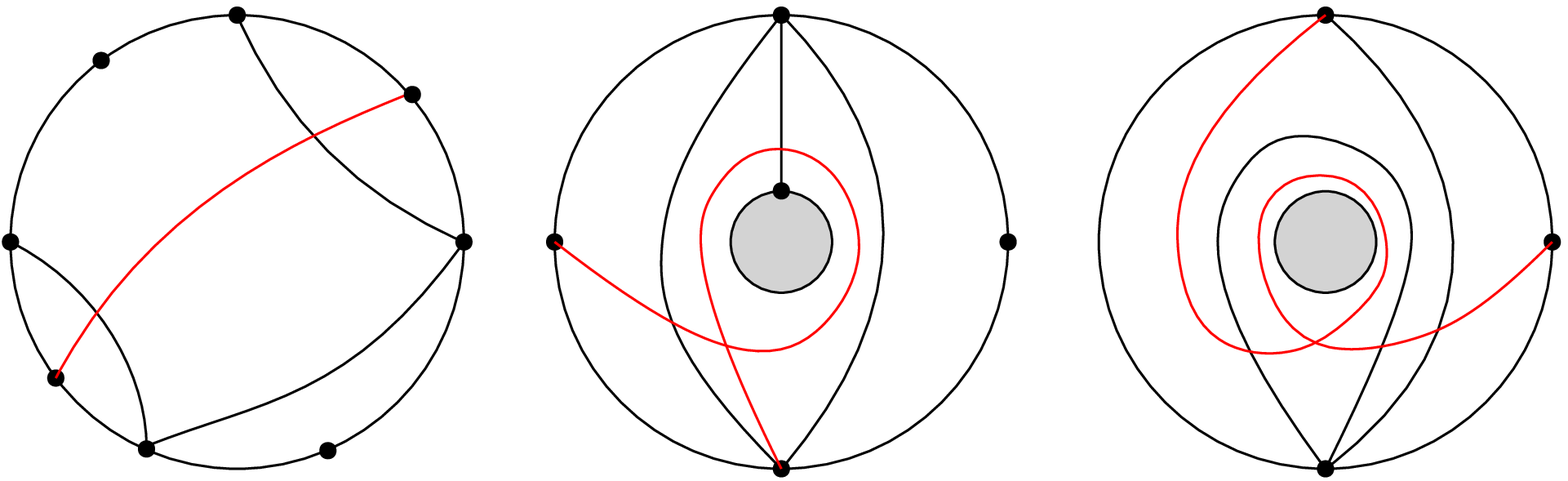}
\caption{Examples of arcs not satisfying condition (3)(b).} 
\label{fig:exnotpermissible2}
\end{figure}
\begin{figure}[H]
\includegraphics[height=3cm]{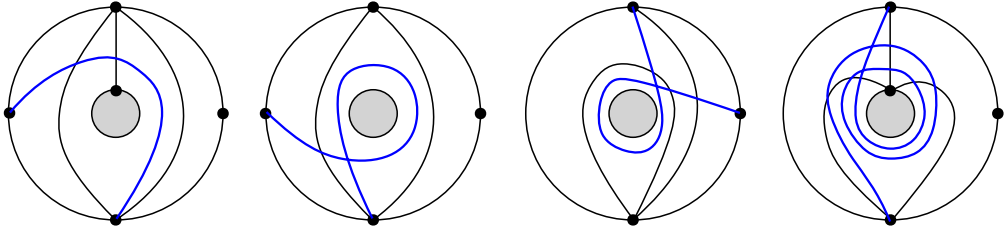}
\caption{Examples of permissible arcs.} 
\label{fig:expermissible}
\end{figure}
\end{example}

\begin{remark}
In the case when $\sP$ is a triangulation of the surface, every arc is permissible. Furthermore, if $\sP$ is a cut of a triangulation, and hence the corresponding algebra is a surface algebra in the sense of~\cite{DRS}, then our notion of permissible arcs coincides with that in~\cite{DRS}. Indeed, condition (3)(a) is superfluous in this case, since there are no unmarked boundary components, and condition (3)(b) boils down to saying that a permissible arc cannot cross two non-adjacent arcs of a quasi-triangle. 
\end{remark}

\begin{definition}
Two permissible arcs $\gamma, \gamma^\prime$ in $(\sS,M)$ are called {\it equivalent}, which we denote by $\gamma \simeq \gamma'$, if one of the following condition holds:
\begin{compactenum}
\item there is a sequence of consecutive sides $\delta_1, \ldots, \delta_k$ of a tile $\Delta$ of $(\sS,\sM,\sP)$ which is not of type I such that:
\begin{compactenum}
\item $\gamma$ is homotopic to the concatenation of $\gamma'$ and $\delta_1, \ldots, \delta_k$, and 
\item $\gamma$ starts at an endpoint of $\delta_1$ ($\delta_k$, resp.), $\gamma'$ starts at an endpoint of $\delta_k$ ($\delta_1$, resp.), and their first crossing with $\sP$ is with the same side of $\Delta$. 
\end{compactenum}
\item the starting points of $\gamma$ and $\gamma'$ are marked points of a tile $\Delta$ of type I or II; their first crossing (at point $p$) with $\sP$ is with the same side of $\Delta$ and the segments of $\gamma$ and $\gamma'$ between $p$ and their ending points are homotopic.
\end{compactenum}
\end{definition}

The equivalence class of a permissible arc $\gamma$ will be denoted by $[\gamma]$. 

\begin{example}
Figures~\ref{fig:equivalence1} and \ref{fig:equivalence2} show examples of permissible arcs which are equivalent. 
\begin{figure}
\includegraphics[height=2.7cm]{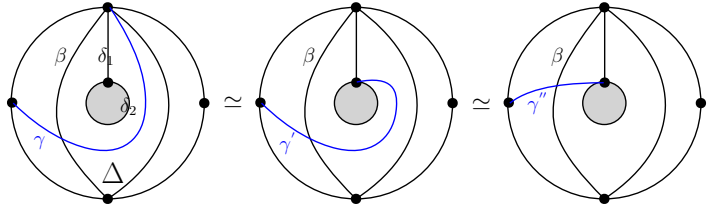}
\caption{Equivalences of permissible arcs.}
\label{fig:equivalence1}
\end{figure}
\begin{figure}
\includegraphics[height=2.7cm]{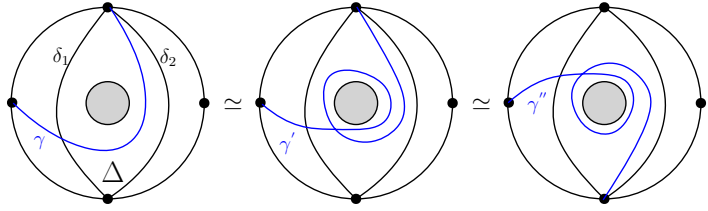}
\caption{Equivalences of permissible arcs when $\Delta$ is of type II.}
\label{fig:equivalence2}
\end{figure}
\end{example}

\begin{remark}
In the case when $\sP$ is a triangulation, equivalence classes of (permissible) arcs coincide with homotopy classes. Furthermore, when $\sP$ is a cut of a triangulation, our notion of equivalence of permissible arcs coincide with that in~\cite{DRS}. Indeed, there are no tiles of type I or II, so condition (2) in our definition does not apply. Moreover, the tiles of $\sP$ are either triangles or quasi-triangles, and condition (1) only applies to the latter, in which case the sequence of consecutive sides must have length one in order to satisfy (1)(b). 
\end{remark}

Given two curves $\gamma, \gamma'$ in $(\sS, M)$, we denote by $I (\gamma, \gamma')$ the minimal number of transversal intersections of representatives of the homotopy classes of $\gamma$ and $\gamma'$. Note that here, we are only considering crossings in the interior of $\sS$, not at its boundary. Hence, we might have $I(\gamma, \gamma') = 0$ for arcs $\gamma$, $\gamma'$ sharing an endpoint. Moreover, if $\gamma \in \sP$, then $I(\gamma, \gamma') = 0$, for every arc $\gamma' \in \sP$. 

Given a curve $\gamma$ in $(\sS, M, \sP)$, the {\it intersection vector} $I_\sP (\gamma)$ of $\gamma$ with respect to $\sP$ is the vector $(I(\gamma, \arc{a}_i))_{\arc{a}_i \in \sP}$.  The {\it intersection number $|I_\sP (\gamma)|$ of $\gamma$ with respect to $\sP$} is given by $\sum_{\arc{a} \in \sP} I(\gamma, \arc{a})$. We say that a permissible arc is {\it trivial} if its intersection number with respect to $P$ is zero. We associate the zero string to any trivial permissible arc $\gamma$, i.e. $w(\gamma) = 0$. Note that this includes boundary segments and arcs in $\sP$. 

\begin{theorem}\label{thm:permissiblearcsstrings}
Let $\sP$ be a partial triangulation of the surface $(\sS, M)$. There is a bijection between the equivalence classes of non-trivial permissible arcs in $(\sS,M)$ 
and the non-zero strings of $A_\sP$. Under this bijection, the intersection vector corresponds to the dimension vector of the corresponding string module. 
\end{theorem}
\begin{proof}
Suppose $w = \xymatrix{v_1 \ar@{<->}[r]^{\alpha_1} & v_2 \ar@{<->}[r]^{\alpha_2} & \cdots \ar@{<->}[r]^{\alpha_{k-1}} & v_k}$ is a non-zero string in $A_{\sP}$. Here, the double headed arrows indicate a fixed but arbitrary orientation of the arrows $\alpha_i$. 

We denote by $\arc{v}_i$ the arc in $\sP$ corresponding to the vertex $v_i$. We define a curve $\gamma (w)$ in $(\sS,\sM)$ as follows:  since there is an arrow $\alpha_1$ between $v_1$ and $v_2$, the arcs $\arc{v}_1, \arc{v}_2$ in $\sP$ share a common endpoint $p \in M$, and there are no other arcs of $\sP$ incident with $p$ sitting between $\arc{v}_1$ and $\arc{v}_2$. Hence, $\arc{v}_1$ and $\arc{v}_2$ are sides of a unique tile $\Delta_1$. We choose one point $q_1$ of the arc $\arc{v}_1$ and one point $q_2$ of the arc $\arc{v}_2$ and connect them by a curve $\gamma_1$ in the interior of $\Delta_1$ in such a way that we have a triangle as in the Figure~\ref{fig:curve}.
\begin{figure}
\includegraphics[height=2.5cm]{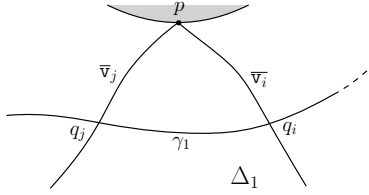}
\caption{The curve $\gamma_1$. Here $i \neq j \in \{1,2\}$.}
\label{fig:curve}
\end{figure}

Proceeding in the same way with the remaining arcs $\arc{v}_2, \ldots, \arc{v}_k$, we obtain curves $\gamma_2, \ldots, \gamma_{k-1}$. Now, the arc $\arc{v}_1$ is incident with to two tiles: $\Delta_1$ considered above and 
$\Delta_0$, the tile on the other side of $\arc{v}_1$. Note that these two tiles may coincide.

{\it Case 1:} $p$ is the only marked point of $\Delta_0$.

This means $\arc{v}_1$ is a loop and there is an unmarked boundary component $\sigma_0$ in the interior of $\Delta_0$, i.e. $\Delta_0$ is of type I. Let $\gamma_0$ be a curve in the interior of $\Delta_0$ with endpoints $p$ and $q_1$ and winding number $1$ around $\sigma_0$. Note that if the winding number were zero, then the concatenation of $\gamma_0$ with $\gamma_1$ would be homotopic to a curve which does not intersect $\arc{v}_1$. 

{\it Case 2:} $\Delta_0$ has exactly two marked points: $p$ and $p'$. 

These points must be the endpoints of $\arc{v}_1$. Since $\arc{v}_1$ is not homotopic to a boundary segment, $\Delta_0$ must be of type II. Choose $\gamma_0$ to be, for instance, a curve in the interior of $\Delta_0$ with endpoints $q_1$ and $p$ in such a way that we have a triangle with vertices $p, q_1$ and $p'$ (see Figure~\ref{fig:firstsegmentII}). 
\begin{figure}[H]
\includegraphics[height=3.3cm]{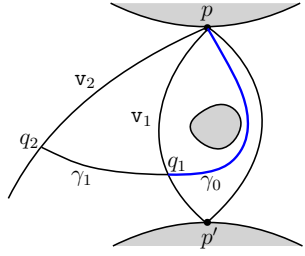}
\caption{$\gamma_0$ when $\Delta_0$ is of type II.}
\label{fig:firstsegmentII}
\end{figure}

{\it Case 3:} $\Delta_0$ has at least three marked points. 

Choose $\gamma_0$ to be a curve in the interior of $\Delta_0$ with endpoints $q_1$ and $q$, where $q$ is one of the marked points in $\Delta_0$ which is not an endpoint of $\arc{v}_1$. 

We proceed in the same way on the other end of the string, constructing a curve $\gamma_k$. The curve $\gamma (w)$ is then defined to be the concatenation of the curves $\gamma_0, \ldots, \gamma_k$. By construction, we have the following:
\begin{compactitem}
\item since $w$ is a string, $v_i \neq v_{i+1}$ unless $\arc{v}_i$ is a loop. Either way, none of the $\gamma_i$ is homotopic to a piece of an arc in $\sP$;
\item the crossing points of $\gamma(w)$ with the arcs in $\sP$ are indexed by the vertices of the string $w$, and the intersections are transversal;
\item $\gamma(w)$ is a permissible arc.  
\end{compactitem}

Thus, the intersection numbers are minimal and $I_{\sP} (\gamma(w)) = \underline{\dim}\, M(w)$, where $\underline{\dim}\, M(w)$ denotes the dimension vector of the module $M(w)$. 

Note that we have a choice of curves $\gamma_0$ and $\gamma_s$. But if we keep the winding number around unmarked boundary components at zero or one, in order to obtain a permissible arc, every possible choice would give a permissible arc equivalent to the one built above. 

Conversely, let $\gamma\colon [0,1] \rightarrow \sS$ be a permissible arc in $(\sS,M)$ with $|I_{\sP} (\gamma)| \neq 0$, belonging to an equivalence class $[\gamma]$. We assume that the arc is chosen (in its homotopy class) such that it intersects the arcs $\arc{a}$ of $\sP$ transversally and such that the intersections are minimal. 

Orienting $\gamma$ from $p = \gamma (0) \in M$ to $q = \gamma (1) \in M$, we denote by $\arc{v}_1$ the first internal arc of $\sP$ that intersects $\gamma$, by $\arc{v}_2$ the second arc, and so on. We thus obtain a sequence $\arc{v}_1, \ldots, \arc{v}_k$ of (not necessarily different) arcs in $\sP$. Since the intersections with arcs in $\sP$ are minimal, we know that $\arc{v}_i \neq \arc{v}_{i+1}$, unless $\arc{v}_i$ is a loop. Either way, there are arrows $\alpha_i\colon v_i \rightarrow v_{i+1}$ or $\alpha_i\colon v_{i+1} \rightarrow v_i$ in $Q (A_{\sP})$, by condition (3)(b) of the definition. Thus, we have a walk $w(\gamma) = \xymatrix{v_1 \ar@{<->}[r]^{\alpha_1} & v_2  \ar@{<->}[r]^{\alpha_2} & \cdots \ar@{<->}[r]^{\alpha_{k-1}} & v_k}$ in $Q (A_{\sP})$, which avoids relations and such that $\alpha_{i+1} \neq \alpha_i^{-1}$. Therefore, $w(\gamma)$ is a non-zero string in $A_{\sP}$. 

If $\gamma' \in [\gamma]$, it follows by the definition of equivalent arcs that $w(\gamma') = w(\gamma)$, so the map from equivalent classes of permissible arcs to strings is well defined. It follows from their construction that the two maps defined above are mutually inverse.   
\end{proof}

To finish the characterisation of indecomposable modules in terms of curves in the surface, we have to consider the bands.

\begin{proposition}\label{prop:bandsclosedcurves}
Let $\sP$ be a partial triangulation of the surface $(\sS, M)$. Then there is a bijection between the homotopy classes of permissible closed curves $c$ in $(\sS, M)$ with $|I_\sP (c)| \geq 2$ and powers of bands of $A_\sP$. Moreover, the permissible closed curve associated to the Auslander-Reiten translate of an indecomposable band module $M$ is the closed curve associated to $M$ itself. 
\end{proposition}
\begin{proof}
The proof of the first statement is similar to that of the previous proposition. The second statement follows from the fact that the Auslander-Reiten translate acts on band modules as the identity morphism.
\end{proof}

Note that there can be permissible closed curves which do not intersect any arc in $\sP$, namely when they go around an unmarked boundary component or intersect just one arc of $\sP$ at only one point. Clearly these curves do not correspond to bands. However, if $\sP$ is a cut of a triangulation, such curves do not occur. This explains why the condition on the intersection number is required in Proposition~\ref{prop:bandsclosedcurves} and why it is not needed in~\cite{DRS}. 

The following result is a straightforward consequence of Proposition~\ref{prop:bandsclosedcurves} and it is a natural generalisation of~\cite[Corollary 5.10]{DRS}. 

\begin{corollary}
A tiling algebra $A_{\sP}$ is of finite representation type if and only if every permissible simple closed curve $c$ is such that $|I_{\sP} (c)| \leq 1$. 
\end{corollary}

\begin{example}
Consider the algebras $\Lambda (r, n, m)$ given by the quiver with relations given in Figure~\ref{fig:DDA}.
\begin{figure}
\includegraphics[height=2.5cm]{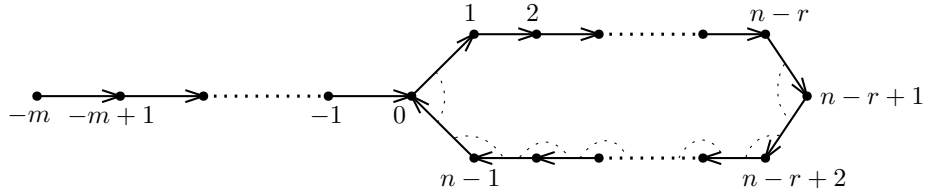}
\caption{The algebra $\Lambda (r, n, m)$.}
\label{fig:DDA}
\end{figure}
Up to derived equivalence, these are all the derived discrete algebras which are not of Dynkin type (cf.~\cite{BGS}), and they are known to be of finite representation type. Such algebras can be seen as tiling algebras associated to the partial triangulations of an annulus given in Figure~\ref{fig:DDAtiling}. 
\begin{figure}
\includegraphics[height=4.5cm]{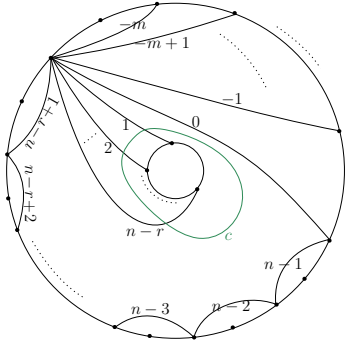}
\caption{Tiling associated to $\Lambda (r, n, m)$.}
\label{fig:DDAtiling}
\end{figure}
As we can see in the figure, either the simple closed curve $c$ is such that $|I_{\sP} (c)| \leq 1$ (if $n-r = 0, 1$) or $|I_{\sP} (c)| \geq 2$ (if $n-r \geq 2$) and it is not permissible. 
\end{example}

\subsection{Pivot elementary moves}
Now we want to give a description of the irreducible morphisms between string modules in terms of pivot elementary moves. 

Given a permissible arc $\gamma$, we denote by $M([\gamma])$ the string module $M(w(\gamma))$, where $w(\gamma)$ is the string associated to $\gamma$. 

\begin{definition}\label{def:pivotmove}
Let $\gamma$ be a permissible arc corresponding to a non-empty string, and let $s$ ($t$, resp.) be its starting point (ending point, resp.). Denote by $\Delta$ the tile in the surface such that the segment $\overline{\gamma}$ of $\gamma$ between $s$ and the first intersection point with $\sP$ lies in the interior of $\Delta$. The {\it pivot elementary move} $f_t([\gamma])$ of $[\gamma]$ is defined to be the equivalence class of the permissible arc $f_t (\gamma_{s})$ obtained from $\gamma$ by fixing the ending point $t$ and moving $s$ in the following way:

{\it Case 1:} $\Delta$ is not of type I or II.

Move $s$ in the counterclockwise direction around $\Delta$ up to the vertex $s'$ for which $\gamma \simeq \gamma_s$, with endpoints $t$ and $s'$, and $\gamma \not\simeq \gamma'$, where $\gamma'$ is homotopic to the concatenation of $\gamma_s$ with the side of $\Delta$ connecting $s'$ with its counterclockwise neighbour. Note that $\gamma_s$ might be $\gamma$ (and so $s' = s$). 

{\it Case 2:} $\Delta$ is of type I.

Let $\gamma_s$ be the permissible arc equivalent to $\gamma$, with starting point $s = s'$ and such that it wraps around the unmarked boundary component in the interior of $\Delta$ in the counterclockwise direction. 

{\it Case 3:} $\Delta$ is of type II. 

Let $\gamma_s$ be the permissible arc equivalent to $\gamma$, with starting point given by one of the two marked points $s'$ of tile $\Delta$, and such that the winding number around the unmarked boundary component $\sigma$ in the interior of $\Delta$ is zero and $\sigma$ is to the left of $\gamma_s$. 

In either case, $f_t (\gamma_{s})$ is defined to be the concatenation of $\gamma_s$ with the boundary segment connecting $s'$ to its counterclockwise neighbour in the boundary component. The pivot elementary move $f_s ([\gamma]) = [f_s(\gamma_t)]$ is defined in a similar way. 
\end{definition}

\begin{theorem}\label{thm:irreduciblepivots}
Let $M(w)$ be a string module over the tiling algebra $A_{\sP}$, with corresponding string $w$. 
\begin{compactenum}
\item Each irreducible morphism in $\mod{A_{\sP}}$ starting at $M(w)$ is obtained 
by a pivot elementary move on an endpoint of the corresponding permissible arc $\gamma = \gamma (w)$.
\item All Auslander-Reiten sequences between string modules in $\mod{A_{\sP}}$ are of the form $0 \rightarrow M([\gamma]) \rightarrow M(f_s(\gamma)) \oplus M(f_t(\gamma)) \rightarrow M(f_s(f_t(\gamma))) \rightarrow 0$, for some permissible arc $\gamma$. 
\end{compactenum}
\end{theorem}
\begin{proof}
By Proposition \ref{prop:BRmaps}, we need to prove that $w_\ell = w(f_t(\gamma_s))$ and $w_r = w(f_s(\gamma_t))$. We will only prove the former, as the latter is similar. 

Firstly, we need to define the starting and ending points $s, t$  of the arc $\gamma$ corresponding to $w$. Fix some sign functions $\sigma, \epsilon$ (recall the definition from Section \ref{sec:gentlebackground}). 

If $w$ is not a trivial string, write $w = \xymatrix{v_1 \ar@{<->}[r]^{\alpha_1} & v_2 \ar@{<->}[r]^{\alpha_2} & \cdots \ar@{<->}[r]^{\alpha_{k-1}} & v_k}$, with $k \geq 2$. Then we  orient 
$\gamma$ in such a way that it crosses the arcs $\arc{v}_1, \ldots, \arc{v}_k$ in this order. 

Now suppose, $w$ is a trivial string, and write $w= 1_v^x$, for some vertex $v$ and $x \in \{+,-\}$. If $v$ is a source, then $M(w)$ is an injective module, and $w_\ell = 0 = w_r$. On the other hand, one can easily see that, independently of the orientation of $\gamma$, the pivot elementary moves at either endpoint give rise to permissible arcs with intersection number zero, which finishes the proof for this case. Thus, suppose there is at least one arrow in $Q$ whose target is $v$. 

Suppose $\arc{v}$ is the side of two different tiles $\Delta_0$ and $\Delta_1$. If there is an arrow $\alpha$ with target $v$ such that $\epsilon (\alpha) = x$ (and so $\alpha w$ is a string), assume, without loss of generality, that the arc corresponding to the source of $\alpha$ is a side of $\Delta_0$. Then orient $\gamma$ such that $s$ is a marked point of $\Delta_0$ and $t$ is a marked point of $\Delta_1$. If there is no such arrow, then there is only one arrow $\beta$ with target $v$ and $\epsilon (\beta) = -x$. In this case, if $s(\beta)$ corresponds to a side of $\Delta_0$, we orient $\gamma$ such that $s \in \Delta_1$ and $t \in \Delta_0$. If the two tiles $\Delta_0$ and $\Delta_1$ coincide, we can use a similar argument to orient the arc $\gamma$. 

 Now that we have determined the orientation of $\gamma$, let $\mathtt{x}_1, \ldots, \mathtt{x}_k$ be the arcs of $\sP$ that $\gamma$ crosses in order. The arc $\mathtt{x}_1$ and the marked point $s$ completely determine a tile of the surface, which shall be denoted by $\Delta$. 

Suppose there is an arrow $\alpha$ in $Q_{A_{\sP}}$ such that $\alpha w$ is a string. Note that this arrow is unique, given the sign functions $\sigma, \epsilon$ and the notion of composition of strings. 

{\it Case 1:} $\Delta$ is not of type I or II. 

By assumption, there is one side $\arc{x}_0$ of $\Delta$, which is not a boundary segment and it has one endpoint in common with $\arc{x}_1$: the marked point corresponding to $\alpha$. Moreover, the arc $\gamma_s$ is the permissible arc obtained from the concatenation of $\gamma$ with the sides of $\Delta$ when travelling counterclockwise around $\Delta$ from $s$ to $s'$, the endpoint of $\arc{x}_0$ not in common with $\arc{x}_1$. Clearly, $\gamma_s \simeq \gamma$. Figure~\ref{fig:pivot-1} 
illustrates this case. 

\begin{figure}[H]
\includegraphics[height=3cm]{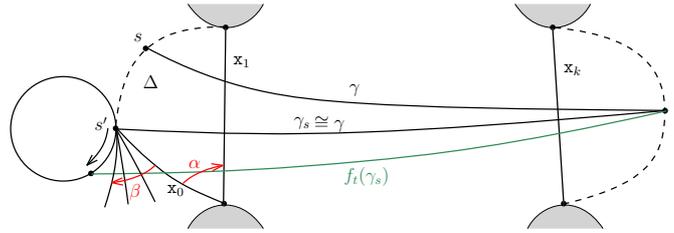}
\caption{Pivot elementary move which adds a hook - Case 1.}\label{fig:pivot-1}
\end{figure}

The arc $f_t (\gamma_s)$ will then cross the fan at $s'$ starting at $\mathtt{x}_0$, 
together with $\mathtt{x}_1, \ldots, \mathtt{x}_k$, as Figure~\ref{fig:pivot-1} shows. 
Hence $f_t (\gamma_s)$ is clearly permissible and the corresponding string is $w(f_t(\gamma_s)) = \beta^{-1} \alpha w$, where $\beta$ is the path corresponding 
to the fan at $s'$ starting at $\mathtt{x}_0$. Thus, $w (f_t(\gamma_s)) = w_\ell$. 

{\it Case 2:} $\Delta$ is of type I.

By construction, the loop arrow at $x_1$ is not the first arrow in string $w$. Given the orientation of $\gamma$, $\alpha$ must be the loop at $x_1$. Let $\gamma_s$ be as in Definition \ref{def:pivotmove}, and let $\arc{y}_1, \ldots, \arc{y}_s$ be the sequence of arcs of $\sP$ incident with $s = s'$ such that there is a string $\beta = x_1 \rightarrow y_1 \rightarrow \cdots \rightarrow y_s$. Thus, $w(f_t(\gamma_s)) = \beta^{-1} \alpha w = w_\ell$ 
(see Figure~\ref{fig:pivot-2}). 
\begin{figure}[H]
\includegraphics[height=3cm]{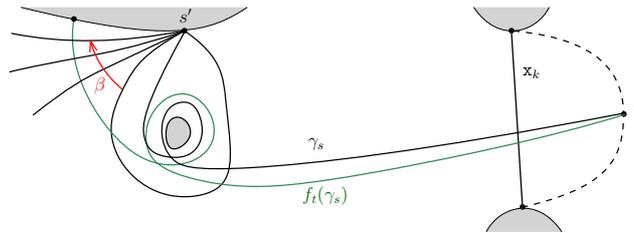}
\caption{Pivot elementary move which adds a hook - Case 2.}\label{fig:pivot-2}
\end{figure}
{\it Case 3:} $\Delta$ is of type II. 

In this case, $\gamma \simeq \gamma_s$, where $\gamma_s$ is as described in Definition \ref{def:pivotmove}. Denote by $x_0$ the other arc of $\Delta$, and let $p_1$ be the starting point of $\gamma_s$ and $p_2$ be the other marked point of $\Delta$. Then $\alpha$ is the arrow from $x_0$ to $x_1$ with corresponding marked point $p_2$. Let $\beta$ be the path associated to the fan at $p_1$ starting at $\arc{x}_0$. Then, $w(f_t(\gamma_s))$ is again $w_\ell$. See Figure~\ref{fig:pivot-3}. 
\begin{figure}[H]
\includegraphics[height=3cm]{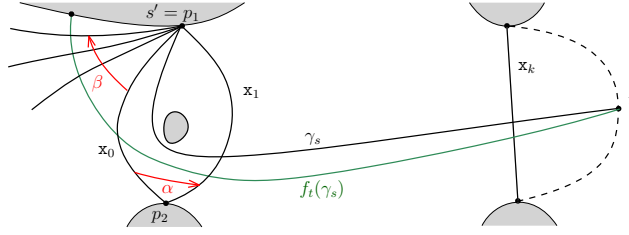}
\caption{Pivot elementary move which adds a hook - Case 3.}\label{fig:pivot-3}
\end{figure}

Finally, suppose there is no arrow $\alpha \in Q_{A_{\sP}}$ such that $\alpha w$ is a string. Then, in particular, $\Delta$ cannot be of type I nor of type II.

Let $p_1$ and $p_2$ be the endpoints of the arc $x_1$ such that $p_1, p_2, s$ and $s'$ follow each other in the counterclockwise order around the border of $\Delta$, where $s'$ is the starting point of $\gamma_s$. By assumption, $\Delta$ has a side with endpoints $s'$ and $p_1$, which is a boundary segment. Thus, $f_t(\gamma_s)$ has endpoints $p_1$ and $t$, and it only crosses $\arc{x}_\ell, \arc{x}_{\ell+1}, \ldots, \arc{x}_k$, where $\arc{x}_\ell$ is the first arc of the sequence $\arc{x}_1, \ldots, \arc{x}_k$ for which the substring $\xymatrix{x_1 \ar@{<->}[r]^{\alpha_1} & x_2 \ar@{<->}[r]^{\alpha_2} & \cdots \ar@{<->}[r]^{\alpha_{\ell-2}} & x_{\ell -1}}$ is direct and $\alpha_{\ell-1}\colon x_{\ell} \rightarrow x_{\ell-1}$. In particular, if $w$ is direct, then $f_t(\gamma_s)$ does not cross any arc in $\sP$. Hence, $w(f_t(\gamma_s)) = w_\ell$. Figure~\ref{fig:pivot-4} illustrates this argument.
\begin{figure}[H]
\includegraphics[height=3.2cm]{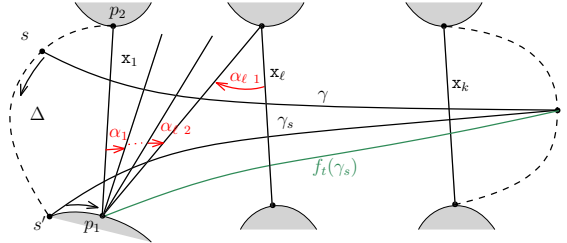}
\caption{Pivot elementary move which removes a cohook.}\label{fig:pivot-4}
\end{figure}

To finish the proof, we need to show that $w(f_s (f_t(\gamma_s))_t) = (w_\ell)_r$. But, by the first part of the proof (and its dual), we have that $w(f_t(\gamma_s)) = w_\ell$ and $w(f_s(f_t(\gamma_s))_t) = (w(f_t(\gamma_s))_t)_r = (w_\ell)_r$. 
\end{proof}

\begin{remark}
Given a permissible arc, denote by $\gamma_{s,t}$ the arc $(\gamma_s)_t$. Note that the construction of $\gamma_s$ ($\gamma_t$, resp.) does not affect the ending point (starting point, resp.) of $\gamma$. Hence $(\gamma_t)_s = (\gamma_s)_t$. Moreover, the ending points (starting points, resp.) of $f_t(\gamma_s)$ and $\gamma_s$ ($f_s(\gamma_t)$ and $\gamma_t$, resp.) coincide. Therefore, $f_t (\gamma_s) = f_t (\gamma_{s,t})$, and $f_s (\gamma_t) = f_s (\gamma_{s,t})$. 
\end{remark}

Given a permissible arc $\gamma$, we define $\tau^{-1} ([\gamma])$ to be the equivalence class of the arc, which we denote by $\tau^{-1} (\gamma_{s,t})$, obtained from $\gamma_{s,t}$ by simultaneous rotation in the counterclockwise direction of its starting and ending points to the next marked points at the boundary. Note that $\tau^{-1} (\gamma_{s,t})$ is again a permissible arc. 

\begin{corollary}
Let $M([\gamma])$ be a string module. We have that $M([\gamma])$ is non-injective if and only if $|I_{\sP} (\tau^{-1} (\gamma_{s,t}))| \neq 0$. Moreover, in this case, we have $\tau^{-1} (M([\gamma])) = M(\tau^{-1} ([\gamma]))$. 
\end{corollary}
\begin{proof}
If follows from the notions of pivot elementary moves that $\tau^{-1} ([\gamma]) = [f_s f_t (\gamma_{s,t})] = [f_t f_s(\gamma_{s,t})]$. By Theorem \ref{thm:irreduciblepivots} (2), we have $M([f_t f_s (\gamma_{s,t})]) = M(\tau^{-1} (M([\gamma])))$, in the case when $M([\gamma])$ is not injective. On the other hand, if $M([\gamma])$ is injective then the corresponding string $w$ is such that there are no arrows $\alpha, \beta$ in $A_\sP$ for which $\alpha w$ or $w \beta$ are strings. One can easily deduce that $w_{r, \ell}$ is the empty string, and that $|I_{\sP} (\tau^{-1} (\gamma_{s,t}))| = 0$.
\end{proof}

\begin{remark}
We need to use the arc $\gamma_{s,t}$ to compute the Auslander-Reiten translate of a module, as 
we may have $[\tau^{-1} (\gamma)] \neq [\tau^{-1} (\gamma_{s,t})]$, 
where $\tau^{-1} (\gamma)$ is defined in the same way as $\tau^{-1} (\gamma_{s,t})$. See 
Figure~\ref{fig:AR-translate} for an example. 

\begin{figure}
\includegraphics[height=5.8cm]{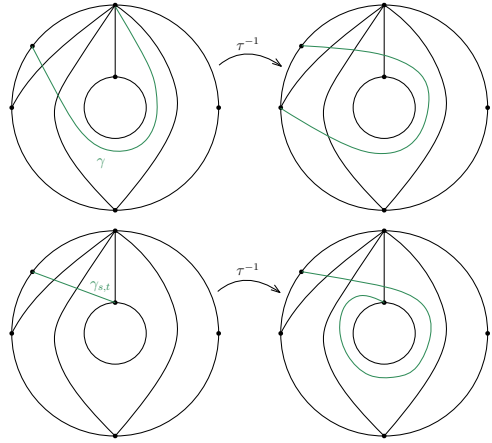}
\caption{Arcs $\gamma$, $\gamma_{s,t}$ and their images under $\tau^{-1}$.}
\label{fig:AR-translate}
\end{figure}
\end{remark}

\subsection{Reading morphisms from curves}

In the previous subsection, we completely described the irreducible morphisms between string modules in terms of pivot elementary moves. However, every morphism between indecomposable modules of a gentle algebra have been previously described~\cite{CB,Krause}. In this subsection we will translate this result in terms of our combinatorial model. 

Let $w$ be a string (which can be a band). A decomposition of $w$ of the form $w=w_1 a^{-1} e b w_2$, where $a, b \in Q_1$, and $w_1, e, w_2$ are strings, is called a {\it factor string}. The set of factor strings of $w$ is denoted by $\Fac (w)$. Similarly, a decomposition of $w$ of the form $w=w_1 c e d^{-1} w_2$, where $c, d \in Q_1$, and $w_1, w_2$ are strings, is called a {\it substring}. The set of substrings of $w$ is denoted by $\Sub (w)$. 

Band modules of the form $M(b, 1, \varphi)$ lie at the mouth of homogeneous tubes, and are therefore called {\it quasi-simple band modules}. 

\begin{theorem}\cite{CB, Krause}\label{thm:homs}
Let v, w be strings or bands, and $M(v), M(w)$ be the corresponding string or quasi-simple band modules. Then $\dim_\kk \, \Hom_A (M(v), M(w)) = |\{(v_1 a^{-1} e b v_2, w_1 c f d^{-1} w_2) \in \Fac (v) \times \Sub (w) \mid f = e \text{ or } f = e^{-1} \}|$.
\end{theorem} 

Pairs in $\Fac (v) \times \Sub (w)$ satisfying the condition stated in the previous theorem are called {\it admissible pairs}. 

\begin{definition}
Let $\gamma$ be a permissible arc or closed curve corresponding to a string or band $w$. Write $\gamma$ as the concatenation of segments $\gamma = \gamma_1 \widetilde{\gamma} \gamma_2$, and let $x$ ($y$, resp.) be the connecting point of $\gamma_1$ ($\gamma_2$, resp.) with $\widetilde{\gamma}$, and assume $x$ lies in the interior of the surface, but does not belong to an arc in $\sP$.  Let $\arc{v}_1, \arc{v}_2, \ldots, \arc{v}_k$ be the arcs in $\sP$ that $\gamma$ cross, in this order. Assume $|I_\sP (\widetilde{\gamma})| \neq 0$ and let $\arc{v}_i, \ldots, \arc{v}_j$ be the arcs crossed by $\widetilde{\gamma}$.
\begin{compactenum}
\item The segment $\tilde{\gamma}$ is said to be a {\it clockwise admissible segment} if it satisfies the following two conditions:  
\begin{compactitem}
\item Either $\arc{v}_i = \arc{v}_1$ or there is an arrow $\alpha\colon v_{i-1} \rightarrow v_i$ in $Q_1$ and a triangle of the form:
\[
\includegraphics[width=4cm]{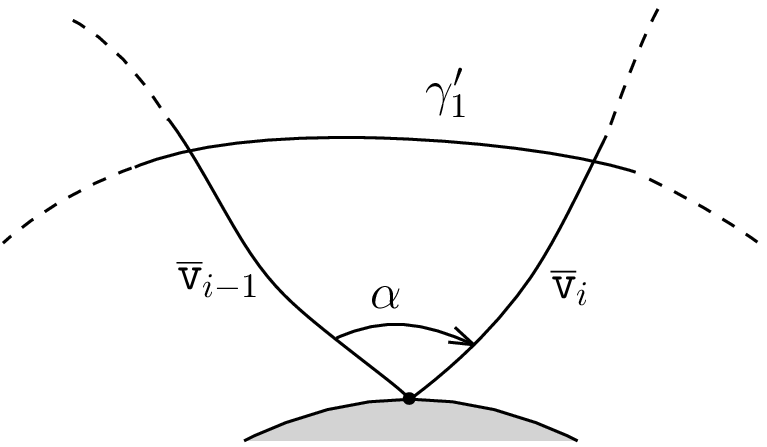}
\]
where $\gamma'_1$ is the segment of $\gamma$ between the intersection points of $\gamma$ with $\arc{v}_{i-1}$ and $\arc{v}_i$. 
\item Either $\arc{v}_j = \arc{v}_k$ or there is an arrow $\beta\colon v_{j+1} \rightarrow v_j$ in $Q_1$ and a triangle of the form:
\[
\includegraphics[width=4cm]{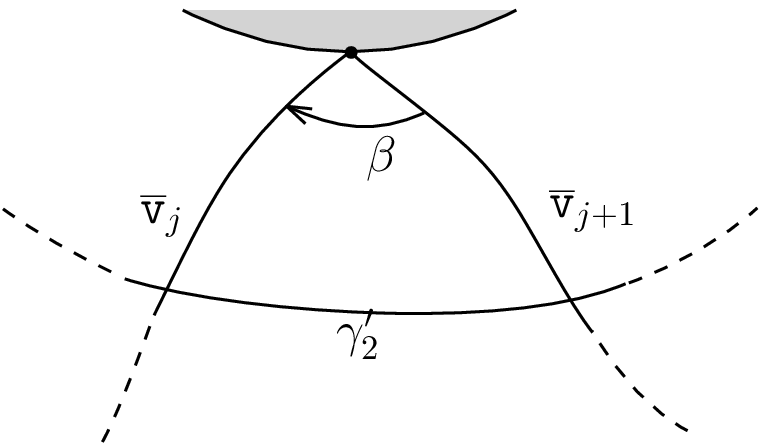}
\]
where $\gamma'_2$ is the segment of $\gamma$ between the intersection points of $\gamma$ with $\arc{v}_{j}$ and $\arc{v}_{j+1}$. 
\end{compactitem}
\item The notion of {\it anticlockwise admissible segment} is obtained from the notion above by changing the orientation of the arrows $\alpha$ and $\beta$. 
\end{compactenum}
\end{definition}

Figure~\ref{fig:admissiblesegment} gives an illustration of this concept. 

\begin{figure}
\includegraphics[height=4.5cm]{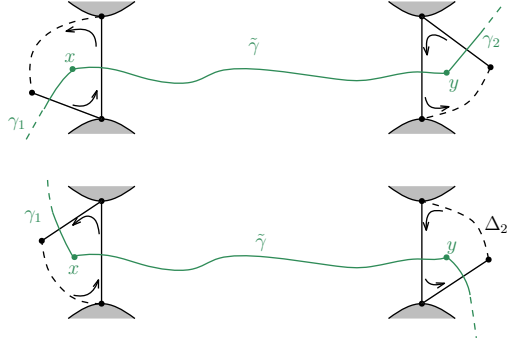}
\caption{Top: $\widetilde{\gamma}$ is clockwise admissible. Bottom: $\widetilde{\gamma}$ is anticlockwise admissible.}
\label{fig:admissiblesegment}
\end{figure}

The next result describes the dimensions of Hom-spaces between two indecomposable modules in terms of their corresponding curves. The proof requires a simple interpretation of Theorem~\ref{thm:homs} in terms of our geometric model and so is omitted.

\begin{proposition}
The dimension of the Hom-space $\Hom_\sA (M(w), M(v))$ is given by the number of pairs $(\widetilde{\gamma (w)}, \widetilde{\gamma (v)})$ of homotopic segments such that $I_\sP (\widetilde{\gamma (w)}) = I_\sP (\widetilde{\gamma (v)})$, $\widetilde{\gamma (w)}$ is anticlockwise admissible and $\widetilde{\gamma (v)}$ is clockwise admissible. 
\end{proposition}

Note that the intersection of $\widetilde{\gamma (w)}$ with $\sP$ gives rise to the string $e$ in Theorem~\ref{thm:homs}. 

\begin{example}
Consider the tiling algebra $\sA = \kk (\xymatrix
{1 \ar@<.3ex>@{->}[r]^a  & 2 \ar@{->}[r]^c \ar@<.3ex>@{->}[l]^b & 3 \ar@(ur,dr)^d  }
)/\langle ab, ba, d^2 \rangle$ given by the tiling in Figure \ref{fig:exampleB}. 
\begin{figure}
\includegraphics[height=3cm]{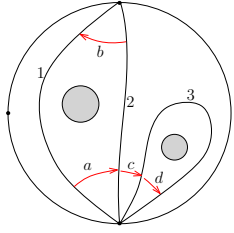}
\caption{Tiling of $\sA$}
\label{fig:exampleB}
\end{figure}

Let $w = b^{-1}cdc^{-1}b$. Then $\dim \, \Hom_\sA (M(w), M(w)) = 2$, as the admissible pairs in $\Fac (w) \times \Sub (w)$ are $(e=w, e= w)$ and $(ed(c^{-1}b), (b^{-1}c)de^{-1})$, where $e = b^{-1}c$ in the latter. Figure \ref{fig:morphismssegments} shows the corresponding homotopic admissible segments. 
\begin{figure}
\includegraphics[height=3cm]{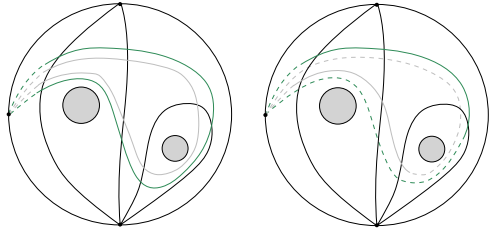}
\caption{Left: $e = w$. Right: $e = b^{-1} c$.}
\label{fig:morphismssegments}
\end{figure}
\end{example}

\begin{remark}
Let $M(w)$ and $M(v)$ be two string modules. If there is a non-zero morphism between $M(v)$ and $M(w)$, then the corresponding arcs $(\gamma (w))_{s,t}$ and $(\gamma (v))_{s,t}$ cross each other, either in the interior of $\sS$ or at one of the endpoints. However, the converse is not true. 
\end{remark}

\subsection{Example}
We finish this section with an example of our model for a representation finite gentle algebra. 
Consider the following tiling algebra: 
\[
\includegraphics[height=3.5cm]{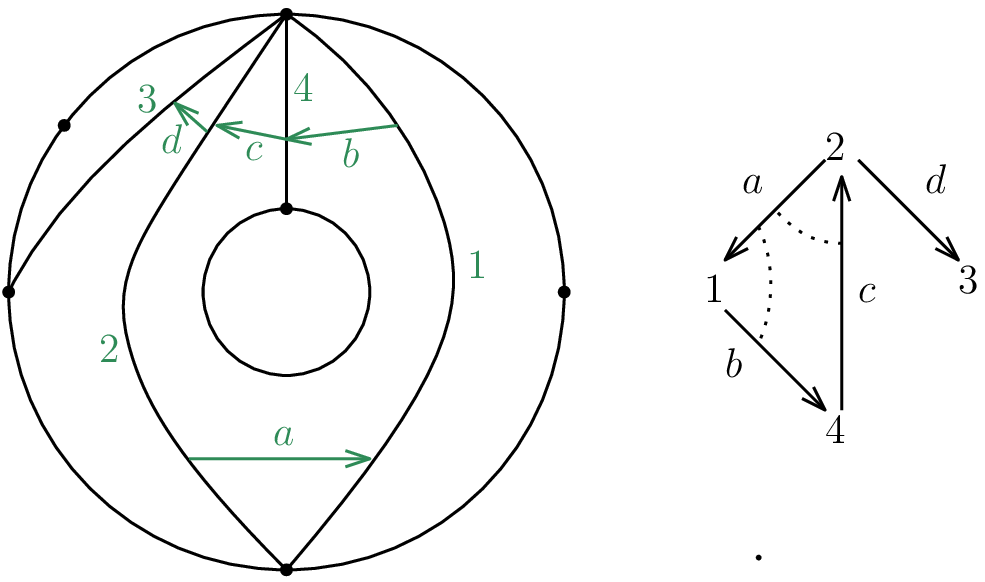}
\]
The tiling algebra is the path algebra over the quiver on the right, subject to the relations $ca=ab=0$. 
The Auslander-Reiten quiver of this gentle algebra is as follows: 
\[
\xymatrix@C-1pc@R-1pc{
 && && bcd\ar[rd] \\
 & && cd\ar@{--}[rr]\ar[rd]\ar[ru] && bc\ar[rd] \\
1\ar@{--}[rr]\ar[rd] && d\ar@{--}[rr]\ar[rd]\ar[ru] && c\ar@{--}[rr]\ar[rd]\ar[ru] && b\ar[rd] \\ 
 & a^{-1}d\ar@{--}[rr]\ar[rd]\ar[ru] && 2\ar@{--}[rr]\ar[ru] && 4\ar@{--}[rr]\ar[ru] && 1 \\
 3\ar@{--}[rr]\ar[ru] && a^{-1}\ar[ru] 
}
\]
In terms of arcs in the annulus, the Auslander-Reiten quiver of the algebra is as in Figure~\ref{fig:ex-curves}, where we chose the arc $\gamma_{s,t}$ as representative for each equivalence class $[\gamma]$.  
\begin{figure}
\includegraphics[width=11cm]{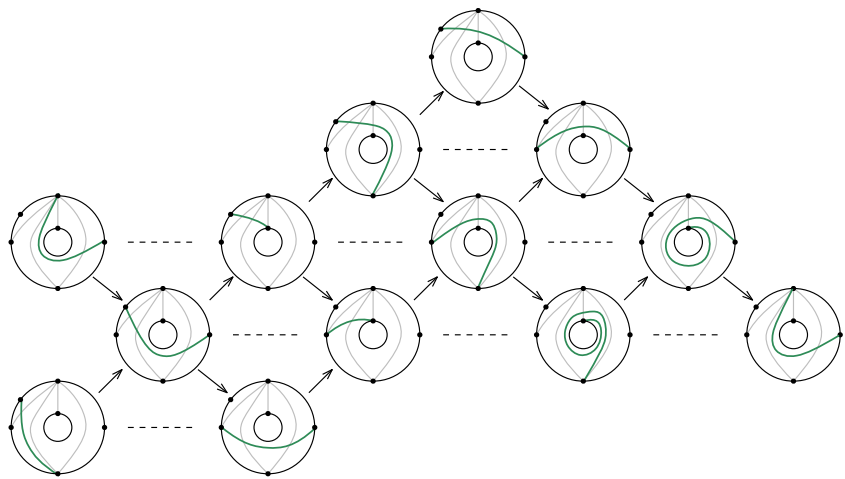}
\caption{Auslander-Reiten quiver in terms of our geometric model.}
\label{fig:ex-curves}
\end{figure}

\subsection*{Acknowledgments} 

RCS thanks Funda\c{c}\~ao para a Ci\^encia e Tecnologia for financial support through Grant SFRH/BPD/90538/2012 and project UID/MAT/04721/2013. 
KB would like to thank the FWF for financial support through grants 30549-N26 and W1230, and University of Leeds for support during a research stay.
Both authors are grateful to Mark J. Parsons for discussions at an early stage of this project. 



\end{document}